\documentclass[12 pt]{article}%
\usepackage{amsmath, amsfonts, amsthm, color,latexsym}
\usepackage{amsmath}
\usepackage[all]{xy}
\usepackage{amsfonts}
\usepackage{amssymb}
\usepackage{color, soul}
\usepackage{graphicx}
\providecommand{\U}[1]{\protect\rule{.1in}{.1in}}
\newtheorem{theorem}{Theorem}[section]
\newtheorem{proposition}[theorem]{Proposition}
\newtheorem{corollary}[theorem]{Corollary}
\newtheorem{example}[theorem]{Example}

\newtheorem{remark}[theorem]{Remark}

\newtheorem{lemma}[theorem]{Lemma}
\newtheorem{final remark}[theorem]{Final Remark}
\newtheorem{definition}[theorem]{Definition}
\begin{document}

\title{Operator ideals related to absolutely summing and Cohen strongly summing operators}
\author{Geraldo Botelho\thanks{Supported by CNPq Grant
305958/2014-3 and Fapemig Grant PPM-00490-15.}\,\,,~ Jamilson R. Campos\thanks{Supported
by a CAPES Postdoctoral scholarship.}~ and Joedson Santos\thanks{Supported by CNPq (Edital Universal 14/2012).\thinspace \hfill\newline\indent2010 Mathematics Subject
Classification: 46B45, 47B10, 47L20.\newline\indent Key words: Banach sequence spaces, operator ideals, summing operators.}}
\date{}
\maketitle

\begin{abstract} We study the ideals of linear operators between Banach spaces determined by the transformation of vector-valued sequences involving the new sequence space introduced by Karn and Sinha \cite{karnsinha} and the classical spaces of absolutely, weakly and Cohen strongly summable sequences. As applications, we prove a new factorization theorem for absolutely summing operators and a contribution to the existence of infinite dimensional spaces formed by non-absolutely summing operators is given.
\end{abstract}

\section*{Introduction and background}

In the theory of ideals of linear operators between Banach spaces (operator ideals), a central role is played by classes of operators that improve the convergence of series, which are usually defined or characterized by the transformation of vector-valued sequences. The most famous of such classes is the ideal of absolutely $p$-summing linear operators, which are the ones that send weakly $p$-summable sequences to absolutely $p$-summable sequences. The celebrated monograph \cite{djt} is devoted to the study of absolutely summing operators.

For a Banach space $E$, let $\ell_p(E)$, $\ell_p^w(E)$ and $\ell_p\langle E \rangle$ denote the spaces of absolutely, weakly and Cohen strongly $p$-summable $E$-valued sequences, respectively. Karn and Sinha \cite{karnsinha} recently introduced a space $\ell_p^{mid}(E)$ of $E$-valued sequences such that
  \begin{equation}\label{inc}
\ell_p\langle E \rangle \subseteq \ell_p(E) \subseteq \ell_p^{mid}(E) \subseteq \ell_p^{w}(E).
\end{equation}
In the realm of the theory of operator ideals, it is a natural step to study the classes of operators $T \colon E \longrightarrow F$ that send: (i) sequences in $\ell_p^w(E)$ to sequences in $\ell_p^{mid}(F)$, (ii) sequences in $\ell_p^{mid}(E)$ to sequences in $\ell_p(F)$, (iii) sequences in $\ell_p^{mid}(E)$ to sequences in $\ell_p\langle F \rangle$. This is the basic motivation of this paper.

We start by taking a closer look at the space $\ell_p^{mid}(E)$ in Section \ref{sec1}. First we give it a norm that makes it a Banach space. Next we consider the relationship with the space $\ell_p^u(E)$ of unconditionally $p$-summable $E$-valued sequences. We show that, although (\ref{inc}) and $\ell_p(E) \subseteq \ell_p^u(E) \subseteq \ell_p^w (E)$ hold for every $E$, in general $\ell_p^{mid}(E)$ and $\ell_p^u(E)$ are not comparable. It is also proved that the correspondence $E \mapsto \ell_p^{mid}(E)$ enjoys a couple of desired properties in the context of operator ideals.

In Section \ref{sec2} we prove that the classes of operators described in (i), (ii) and (iii) above are Banach operator ideals. Characterizations of each class and their corresponding norms are given and properties of each ideal are proved. We establish a factorization theorem for absolutely summing operators and a question left open in \cite{karnsinha} is settled. In both sections \ref{sec1} and \ref{sec2} we study Banach spaces $E$ for which $\ell_p(E) = \ell_p^{mid}(E)$ or $\ell_p^{mid}(E) = \ell_p^w(E)$.

In Section \ref{sec3} we give an application to the existence of infinite dimensional Banach spaces formed, up to the null operator, by non-absolutely summing linear operators on non-superreflexive spaces.

Let us define the classical sequences spaces we shall work with:\\
$\bullet$ $\ell_p(E) $ = absolutely $p$-summable $E$-valued sequences with the usual norm $\|\cdot\|_p$;\\
$\bullet$ $\ell_p^w(E)$ = weakly $p$-summable $E$-valued sequences with the norm
$$\|(x_j)_{j=1}^\infty\|_{w,p} = \sup_{x^* \in B_{E*}}\|(x^*(x_j))_{j=1}^\infty\|_p; $$
$\bullet$ $\ell_p^u(E) = \left\{(x_j)_{j=1}^\infty \in \ell_p^w(E) : \displaystyle\lim_k \|(x_j)_{j=k}^\infty\|_{w,p} = 0 \right\}$ with the norm inherited from $\ell_p^w(E)$ (unconditionally $p$-summable sequences, see \cite[8.2]{df});\\
$\bullet$ $\displaystyle\ell_p\langle E \rangle  = \left\{(x_j)_{j=1}^\infty \in E^{\mathbb{N}}: \|(x_j)_{j=1}^\infty\|_{C,p}:= \sup_{(x_j^*)_{j=1}^\infty \in B_{\ell_{p^*}^w(E^*)}} \|(x_j^*(x_j))_{j=1}^\infty\|_1< +\infty\right\}$,\\ where $\frac{1}{p} + \frac{1}{p^*} = 1$,  (Cohen strongly $p$-summable sequences or strongly $p$-summable sequences, see, e.g., \cite{cohen73}).

The letters $E,F$ shall denote Banach spaces over $\mathbb{K} = \mathbb{R}$ or $\mathbb{C}$. The closed unit ball of $E$ is denoted by $B_E$ and its topological dual by $E^*$. The symbol $E \stackrel{1}{\hookrightarrow} F$ means that $E$ is a linear subspace of $F$ and $\|x\|_F \leq \|x\|_E$ for every $x \in E$. By ${\cal L}(E;F)$ we denote the Banach space of all continuous linear operators $T \colon E \longrightarrow F$ endowed with the usual sup norm. By $\Pi_{p;q}$  we denote the ideal of absolutely $(p;q)$-summing linear operators \cite{df,djt}. If $p=q$ we simply write $\Pi_{p}$. The ideal of Cohen strongly $p$-summing linear operators \cite{jamilson13,cohen73} shall be denoted by ${\cal D}_{p}$. We use the standard notation of the theory of operator ideals \cite{df, pietschlivro}.

\section{The space $\ell_p^{mid}(E)$}\label{sec1}

In this section we give the space of sequences defined in Karn and Sinha \cite{karnsinha} a norm that makes it a Banach space and establish some useful properties of this space.

The vector-valued sequences introduced in \cite{karnsinha} are called \emph{operator $p$-summable sequences}. This term is quite inconvenient for our purposes, and considering the intermediate position of the space formed by such sequences between $\ell_p(E)$ and $\ell_p^w(E)$ (see (\ref{inc})), we shall use the term \emph{mid $p$-summable sequences}. Instead of the original definition, we shall use a characterization proved in \cite[Lemma 2.3, Proposition 2.4]{karnsinha}:

\begin{definition}\rm A sequence $(x_j)_{j=1}^\infty$ in a Banach space $E$ is said to be \emph{mid $p$-summable}, $1 \leq  p  <  \infty$, if $(x_j)_{j=1}^\infty \in \ell_p^w(E)$ and $\left(\left(x_n^*\left(x_j\right)\right)_{j=1}^\infty\right)_{n=1}^\infty \in \ell_p\left(\ell_p\right)$ whenever $(x_n^*)_{n=1}^\infty \in \ell_p^{w}(E^{*})$. The space of all such sequences shall be denoted by $\ell_p^{mid}(E)$.
\end{definition}

The following {\it extreme} cases will be important throughout the paper:

\begin{theorem}{\rm \cite[Proposition 3.1 and Theorem 4.5]{karnsinha})} \label{coinc} Let $E$ be a Banach space and $1 \leq p < \infty$. Then:\\
{\rm (i)} $\ell_p^{mid}(E) = \ell_p^{w}(E)$ if and only if $\Pi_p(E;\ell_p) = {\cal L}(E;\ell_p)$.\\
{\rm (ii)} $\ell_p^{mid}(E) = \ell_p(E)$ if and only if $E$ is a subspace of $L_p(\mu)$ for some Borel measure $\mu$.
\end{theorem}

We say that a Banach space $E$ is a {\it weak mid $p$-space} if $\ell_p^{mid}(E) = \ell_p^{w}(E)$; and it is a {\it strong mid $p$-space} if $\ell_p^{mid}(E) = \ell_p(E)$.

The space $\ell_p^{mid}(E)$ is not endowed with a norm in \cite{karnsinha}. Our first goal in this section is to give it a useful complete norm. Let us see first that the norm inherited from $\ell_p^w(E)$ is helpless. We believe the next lemma is folklore; we give a short proof because we have found no reference to quote. As usual, $c_{00}(E)$ means the space of finite (or eventually null) $E$-valued sequences.

\begin{lemma}\label{nnll} If $E$ is infinite dimensional, then the norms $\|\cdot\|_p$ and $\|\cdot\|_{w,p}$ are not equivalent on $c_{00}(E)$. In particular, $\ell_p(E)$ is not closed in $\ell_p^w(E)$.
\end{lemma}

\begin{proof} It is clear that $c_{00}(E)$ is dense in $(\ell_p(E), \|\cdot\|_p)$, and the definition of $\ell_p^u(E)$ makes clear that $c_{00}(E)$ is dense in $(\ell_p^u(E), \|\cdot\|_{w,p})$ as well. Assume that the norms $\|\cdot\|_p$ and $\|\cdot\|_{w,p}$ are equivalent on $c_{00}(E)$.
Then
$$ \ell_p(E) =\overline{c_{00}(E)}^{\|\cdot\|_{p}} = \overline{c_{00}(E)}^{\|\cdot\|_{w,p}} = \ell_p^u(E). $$
It follows that the identity operator in $E$ is absolutely $p$-summing \cite[Proposition 11.1(c)]{df}, hence $E$ is finite dimensional. Now the second assertion follows from the Open Mapping Theorem and the inclusion $c_{00}(E) \subseteq \ell_p(E)$.
\end{proof}

From Theorem \ref{coinc}(ii) we know that $\ell_p^{mid}(\ell_p) = \ell_p(\ell_p)$, so  $\ell_p^{mid}(\ell_p)$ \emph{is not} closed in $\ell_p^{w}(\ell_p)$ by Lemma \ref{nnll}, proving that $\|\cdot\|_{w,p}$ does not make $\ell_p^{mid}(E)$ complete in general.

\begin{proposition} The expression
\begin{equation}\label{normmid}
\|(x_j)_{j=1}^\infty\|_{mid,p} := \sup_{(x_n^*)_{n=1}^\infty \in B_{\ell_p^w(E^{*})}} \left(\sum_{n=1}^\infty \sum_{j=1}^\infty |x_n^*(x_j)|^p\right)^{1/p},
\end{equation}
is a norm that makes $\ell_p^{mid}(E)$ a Banach space and
$\ell_p(E) \stackrel{1}{\hookrightarrow} \ell_p^{mid}(E) \stackrel{1}{\hookrightarrow} \ell_p^{w}(E).$
\end{proposition}

\begin{proof} Let $x=(x_j)_{j=1}^\infty \in \ell_p^{mid}(E)$. By definition, the double series in (\ref{normmid}) is convergent (this is why we chose this condition to be the definition of $\ell_p^{mid}(E)$). The map
$$T_x \colon \ell_p^w(E^{*}) \longrightarrow \ell_p(\ell_p)~,~T_x((x_n^*)_{n=1}^\infty) = \left(\left(x_n^*\left(x_j\right)\right)_{j=1}^\infty\right)_{n=1}^\infty$$
is a well defined linear operator. By the Closed Graph Theorem, it
is continuous. Therefore,
\[\left(\sum_{n=1}^\infty \sum_{j=1}^\infty |x_n^*(x_j)|^p\right)^{1/p} = \|T_x((x_n^*)_{n=1}^\infty)\| \leq \|T_x\| \cdot\|(x_n^*)_{n=1}^\infty\|_{w,p} \]
for every $(x_n^*)_{n=1}^\infty \in {\ell_p^w(E^{*})}$, showing that the supremum in (\ref{normmid}) is finite. Straightforward computations prove that $\|\cdot\|_{mid,p}$ is a norm and a canonical argument shows that $\left(\ell_p^{mid}(E),\|\cdot\|_{mid,p}\right)$ is a Banach space.

For every $\varphi \in B_{E^{*}}$, it is clear that $(\varphi,0,0,\ldots) \in B_{\ell_p^w(E^{*})}$, so $\|\cdot \|_{w,p} \leq \|\cdot\|_{mid,p}$ in $\ell_p^{mid}(E)$.

Let $(x_j)_{j=1}^\infty \in \ell_p(E)$ and $(x_n^*)_{n=1}^\infty \in \ell_p^w(E^{'})$. Using that $B_E$, regarded as a subspace of $E^{**}$, is a norming subset of $E^{**}$, we have $\|(x_n^{*})_{n=1}^\infty\|_{w,p}^p = \sup\limits_{x \in B_E} \sum\limits_{n=1}^\infty |x_n^*(x)|^p$. Putting $J = \{j \in \mathbb{N} : x_j \neq 0\}$, we have
\begin{align*}\sum_{j=1}^\infty \sum_{n=1}^\infty |x_n^*(x_j)|^p &=\sum_{j \in J}\left(\|x_j\|^p \cdot \left( \sum_{n=1}^\infty \left|x_n^*\left(\frac{x_j}{\|x_j\|}\right)\right|^p\right)\right) \\&\leq \|(x_n^*)_{n=1}^\infty\|_{w,p}^p\cdot \sum_{j\in J} \|x_j\|^p = \|(x_n^*)_{n=1}^\infty\|_{w,p}^p\cdot \sum_{j=1}^\infty \|x_j\|^p,
\end{align*}
from which the inequality
$\|\cdot\|_{mid,p} \leq \|\cdot\|_{p}$ follows.
\end{proof}

\begin{proposition} The following are equivalent for a weak mid-$p$-space $E$:\\
{\rm (a)} $\ell_p(E)$ is closed in $\ell_p^{mid}(E)$.\\
{\rm (b)} The norms $\|\cdot\|_p$ and $\|\cdot\|_{mid,p}$ are equivalent on $\ell_p(E)$.\\
{\rm (c)} The norms $\|\cdot\|_p$ and $\|\cdot\|_{mid,p}$ are equivalent on $c_{00}(E)$.\\
{\rm (d)} $E$ is finite dimensional.
\end{proposition}
\begin{proof} (a) $\Longrightarrow$ (b) follows from the Open Mapping Theorem, (b) $\Longrightarrow$ (c) and (d) $\Longrightarrow$ (a) are obvious. Let us prove (c) $\Longrightarrow$ (d): Since $E$ is a weak mid $p$-space, the norms $\|\cdot\|_{mid,p}$ and $\|\cdot\|_{w,p}$ are equivalent on $\ell_p^{mid}(E)$ by the Open Mapping Theorem, hence they are equivalent on $c_{00}(E)$. The assumption gives that the norms $\|\cdot\|_p$ and $\|\cdot\|_{w,p}$ are equivalent on  $c_{00}(E)$. By Lemma \ref{nnll} it follows that $E$ is finite dimensional.
\end{proof}
Analogously, we have:
\begin{proposition} The following are equivalent for a strong mid-$p$-space $E$:\\
{\rm (a)} $\ell_{mid}(E)$ is closed in $\ell_p^{w}(E)$.\\
{\rm (b)} The norms $\|\cdot\|_{mid,p}$ and $\|\cdot\|_{w,p}$ are equivalent on $\ell_p^{mid}(E)$.\\
{\rm (c)} The norms $\|\cdot\|_{mid,p}$ and $\|\cdot\|_{w,p}$ are equivalent on $c_{00}(E)$.\\
{\rm (d)} $E$ is finite dimensional.
\end{proposition}

The next examples show that the spaces $\ell_p^u(E)$ and $\ell_p^{mid}(E)$ are incomparable in general.

\begin{example}\rm On the one hand, combining Theorem \ref{coinc}(i) with \cite[Theorem 3.7]{djt} we have $\ell_2^{mid}(c_0) = \ell_2^w(c_0)$. Since $\ell_2^u(c_0)$ is a proper subspace of $\ell_2^w(c_0)$ \cite[p.\,93]{df}, it follows that $\ell_2^{mid}(c_0) \nsubseteqq \ell_2^u(c_0)$. On the other hand,
$$\ell_1^u(\ell_1) = \ell_1^w(\ell_1) \nsubseteqq \ell_1(\ell_1) = \ell_1^{mid}
(\ell_1),$$
where the first equality follows from the fact that bounded linear operators from $c_0$ to $\ell_1$ are compact combined with  \cite[Proposition 8.2(1)]{df}, and the last equality is a consequence of Theorem \ref{coinc}(ii).
\end{example}

We saw that $\ell_p^{mid}(E)$ is not contained in $\ell_p^u(E)$ in general. But sometimes this happens:
\begin{proposition} If $E$ is a strong mid $p$-space, then $\ell_p^{mid}(E) \stackrel{1}{\hookrightarrow} \ell_p^u(E)$.
\end{proposition}
\begin{proof} The norms $\|\cdot\|_p$ and $\|\cdot\|_{mid,p}$ are equivalent on $\ell_p^{mid}(E) = \ell_p(E)$ by the Open Mapping Theorem. Let $x = (x_j)_{j=1}^\infty \in \ell_p^{mid}(E)$. Since $(x_j)_{j=1}^k \stackrel{k~~~}{\longrightarrow x}$ in $\ell_p(E)$, by the equivalence of the norms we have $(x_j)_{j=1}^k \stackrel{k~~~}{\longrightarrow x}$ in $\ell_p^{mid}(E)$. As $\ell_p^{mid}(E) \stackrel{1}{\hookrightarrow} \ell_p^w(E)$, we have
$$ \|(x_j)_{j=k}^\infty \|_{w,p}  = \|(x_j)_{j=1}^\infty - (x_j)_{j=1}^{k-1}\|_{w,p} \leq \|(x_j)_{j=1}^\infty - (x_j)_{j=1}^{k-1}\|_{mid,p} \stackrel{k\rightarrow \infty}{\longrightarrow}0, $$
proving that $x \in \ell_p^u(E)$.
\end{proof}

The purpose of the next section is to study the operator ideals determined by the transformation of vector-valued sequences belonging to the sequence spaces in the chain (\ref{inc}). A usual approach, proving all the desired properties using the definitions of the underlying sequence spaces, would lead to long and boring proofs. Alternatively, we shall apply the abstract framework constructed in  \cite{botelhocampos} to deal with this kind of operators. In this fashion we will end up with short and concise proofs. Instead of its definition, we shall use that the class of mid $p$-summable sequences enjoys the two properties we prove below. For the definitions of finitely determined and linearly stable sequence classes, see \cite{botelhocampos}.

\begin{proposition}\label{seqclass} The correspondence
$E \mapsto \ell_p^{mid}(E)$ is a finitely determined sequence class.
\end{proposition}

\begin{proof}
It is plain that $c_{00}(E) \subseteq \ell_p^{mid}(E)$ and
$\|e_j\|_{mid,p} = 1$, where $e_j$ is the $j$-th canonical unit
scalar-valued sequence. Since $\ell_p^{mid}(E)
\stackrel{1}{\hookrightarrow} \ell_p^w(E)$ and $\ell_p^{w}(E)
\stackrel{1}{\hookrightarrow} \ell_\infty(E)$, we have
$\ell_p^{mid}(E) \stackrel{1}{\hookrightarrow} \ell_\infty(E)$. Let
$(x_j)_{j=1}^\infty$ be a $E$-valued sequence. The equality
$$\sup_{(x_n^*)_{n=1}^\infty \in  B_{\ell_p^w(E^{*})}}\left(\sum_{n=1}^\infty \sum_{j=1}^\infty |x_n^*(x_j)|^p\right)^{1/p} = \sup_k \sup_{(x_n^*)_{n=1}^\infty \in  B_{\ell_p^w(E^{*})}}\left(\sum_{n=1}^\infty \sum_{j=1}^k |x_n^*(x_j)|^p\right)^{1/p} $$
shows that $(x_j)_{j=1}^\infty \in \ell_p^{mid}(E)$ if and only if $\sup\limits_k \|(x_j)_{j=1}^k\|_{mid,p} < \infty$ and that
$\|(x_j)_{j=1}^\infty\|_{mid,p} = \sup\limits_k \|(x_j)_{j=1}^k\|_{mid,p}$.
\end{proof}

\begin{proposition}\label{linestab}
The correspondence
$E \mapsto \ell_p^{mid}(E)$ is linearly stable.
\end{proposition}

\begin{proof}
Let $T \in {\cal L}(E;F)$. By the linear stability of $\ell_p^{w}(\cdot)$  \cite[Theorem 3.3]{botelhocampos}, $(T^*(y_n^*))_{n=1}^\infty = (y_n^* \circ T)_{n=1}^\infty\in \ell_p^w(E^{*})$ for every $(y_n^*)_{n=1}^\infty \in \ell_p^w(F^{*})$, where $T^{*}\colon F^{*} \longrightarrow E^{*}$ is the adjoint of $T$. Therefore, 
\[\left(\left(y_n^*(T\left(x_j\right))\right)_{j=1}^\infty\right)_{n=1}^\infty = \left(\left(y_n^* \circ T\left(x_j\right)\right)_{j=1}^\infty\right)_{n=1}^\infty \in \ell_p\left(\ell_p\right),\]
hence $(T(x_j))_{j=1}^\infty \in \ell_p^{mid}(F)$ for every $(x_j)_{j=1}^\infty \in \ell_p^{mid}(E)$. Defining $\widehat{T}\colon \ell_p^{mid}(E) \longrightarrow \ell_p^{mid}(F)$ by $\widehat{T}((x_j)_{j=1}^\infty) = (T(x_j))_{j=1}^\infty$, a standard calculation shows that $\|T\|= \|\widehat{T}\|$.
\end{proof}

\section{Mid summing operators}
\label{sec2}

Following the classical line of studying operators that improve the
summability of sequences, in this section we investigate the obvious
classes of operators, involving mid $p$-summable sequences,
determined by the chain
$$\ell_p \langle E \rangle  \subseteq \ell_p(E) \subseteq \ell_p^{mid}(E) \subseteq \ell_p^{w}(E). $$

From now on in this section, $1 \leq q \leq p <\infty$ are real numbers and $T \in \mathcal{L}(E;F)$ is a continuous linear operator.

\begin{definition}\rm \label{defop}
The operator $T$ is said to be:\\
{\rm (i)} \emph{Absolutely mid $(p;q)$-summing} if
\begin{equation}\label{aps}
\left(T\left(x_j\right)\right)_{j=1}^\infty \in \ell_p(F)\ \
\mathrm{whenever}\ \ (x_j)_{j=1}^\infty \in \ell_{q}^{mid}(E).
\end{equation}
{\rm (ii)} \emph{Weakly mid $(p;q)$-summing} if
\begin{equation}\label{wps}
\left(T\left(x_j\right)\right)_{j=1}^\infty \in \ell_p^{mid}(F)\ \
\mathrm{whenever}\ \ (x_j)_{j=1}^\infty \in \ell_{q}^{w}(E).
\end{equation}
{\rm (iii)} \emph{Cohen mid $p$-summing} if
\begin{equation}\label{cps}
\left(T\left(x_j\right)\right)_{j=1}^\infty \in \ell_p\langle F
\rangle \ \ \mathrm{whenever}\ \ (x_j)_{j=1}^\infty \in
\ell_{p}^{mid}(E).
\end{equation}
\end{definition}

The spaces formed by the operators above shall be denoted by $\Pi_{p;q}^{mid}(E;F), W_{p;q}^{mid}(E;F)$ and ${\cal D}_{p}^{mid}(E;F)$, respectively. When $p =  q$ we simply write mid $p$-summing instead mid $(p;p)$-summing and use the symbols  $\Pi_{p}^{mid}$ and $W_{p}^{mid}$. A standard calculation shows that if  $p < q$ then $\Pi_{p;q}^{mid}(E;F) =  W_{p;q}^{mid}(E;F) = \{0\}$. From the definitions it is clear that
$$\Pi_{p;q}= W_{p;q}^{mid} \cap \Pi_{p;q}^{mid} {\rm ~~and~~}{\cal D}_{p}^{mid}= {\cal D}_p \cap \Pi_{p}^{mid}.$$

Having in mind the properties of $\ell_p^{mid}(E)$ proved in the
previous section, the following three results are straightforward
consequences of \cite[Proposition 1.4]{botelhocampos} (with the
exception of the equivalences in Theorem \ref{teoweak} involving
$\ell_p^u(E)$, which follow from \cite[Corollary
1.6]{botelhocampos}). Recall that any map $T \colon E
\longrightarrow F$ induces a map $\widetilde{T}$ between $E$-valued
sequences and $F$-valued sequences given by
$\widetilde{T}\left((x_j)_{j=1}^\infty \right) =
(T(x_j))_{j=1}^\infty$.

\begin{theorem} The following are equivalent: \\
{\rm (i)} $T\in \Pi_{p;q}^{mid}(E;F)$.\\
{\rm (ii)} The induced map $\widetilde{T}\colon \ell_q^{mid}(E) \longrightarrow \ell_p(F)$ is a well defined continuous linear operator.\\
{\rm (iii)} There is a constant $A> 0$ such that $\left\|\left(T\left(x_j\right)\right)_{j=1}^k\right\|_p \leq A \left\|(x_j)_{j=1}^k\right\|_{mid,q}$ for every $k \in \mathbb{N}$ and all $x_j \in E$, $j=1,\ldots,k$.\\
{\rm (iv)} There is a constant $A > 0$ such that  $\left\|\left(T\left(x_j\right)\right)_{j=1}^\infty\right\|_p \leq A  \left\|(x_j)_{j=1}^\infty\right\|_{mid,q}$  for every $\left(x_j\right)_{j=1}^\infty \in \ell_{q}^{mid}(E)$.

Moreover,
$$\|T\|_{\Pi_{p;q}^{mid}}:= \|\widetilde{T}\| = \inf\{A: {\rm (iii)~holds}\} = \inf\{A: {\rm (iv)~holds}\}.$$
\end{theorem}

\begin{theorem}\label{teoweak} The following are equivalent: \\
{\rm (i)} $T\in W_{p;q}^{mid}(E;F)$.\\
{\rm (ii)} The induced map $\widetilde{T}\colon \ell_q^{w}(E) \longrightarrow \ell_p^{mid}(F)$ is a well defined continuous linear operator.

\vspace*{0.4em}

\noindent{\rm (iii)} $\left(T\left(x_j\right)\right)_{j=1}^\infty \in \ell_p^{mid}(F)$ whenever $(x_j)_{j=1}^\infty \in \ell_{q}^{u}(E)$.\\
{\rm (iv)} The induced map $\widehat{T}\colon \ell_q^{u}(E) \longrightarrow \ell_p^{mid}(F)$ is a well defined continuous linear operator.
{\rm (v)} There is a constant $B >0$ such that $\left\|\left(T\left(x_j\right)\right)_{j=1}^k\right\|_{mid,p} \leq B \left\|(x_j)_{j=1}^k\right\|_{w,q}$ for every $k \in \mathbb{N}$ and all $x_j \in E$, $j=1,\ldots,k$.\\
{\rm (vi)}  There is a constant $B >0$ such that $$\left(\sum\limits_{n=1}^\infty \sum\limits_{j=1}^k \left|y_n^*\left(T\left(x_j\right)\right)\right|^p\right)^{1/p}\leq B \left\|(x_j)_{j=1}^k\right\|_{w,q}\cdot\left\|(y_n^*)_{n=1}^\infty \right\|_{w,p}$$ for every $k \in \mathbb{N}$, all $x_j \in E$, $j=1,\ldots,k$, and every $(y_n^*)_{n=1}^\infty \in \ell_{p}^{w}(F^*)$.\\
  {\rm (vii)}  There is a constant $B >0$ such that  $$\left(\sum\limits_{n=1}^\infty \sum\limits_{j=1}^\infty \left|y_n^*\left(T\left(x_j\right)\right)\right|^p\right)^{1/p}\leq B \left\|(x_j)_{j=1}^\infty\right\|_{w,q}\cdot\left\|(y_n^*)_{n=1}^\infty\right\|_{w,p}$$ for all $\left(x_j\right)_{j=1}^\infty \in \ell_{q}^{w}(E)$ and $(y_n^*)_{n=1}^\infty \in \ell_{p}^{w}(F^*)$.

Moreover,
$$\|T\|_{W_{p;q}^{mid}}:= \|\widetilde{T}\| = \|\widehat{T}\| = \inf\{B: {\rm (v)~holds}\} = \inf\{B: {\rm (vi)~holds}\}= \inf\{B: {\rm (vii)~holds}\}.$$
\end{theorem}

\begin{theorem} The following are equivalent: \\
{\rm (i)} $T\in {\cal D}_{p}^{mid}(E;F)$.\\
{\rm (ii)} The induced map $\widetilde{T}\colon \ell_p^{mid}(E) \longrightarrow \ell_p\langle F\rangle$ is a well defined continuous linear operator.
{\rm (iii)} There is a constant $C>0$ such that  $\left\|\left(T\left(x_j\right)\right)_{j=1}^k\right\|_{C,p} \leq C \left\|(x_j)_{j=1}^k\right\|_{mid,p}$ for every $k \in \mathbb{N}$ and all $x_j \in E$, $j=1,\ldots,k$.\\
 {\rm (iv)} There is a constant $C>0$ such that   $$\sum_{j=1}^k \left|y_j^*\left(T\left(x_j\right)\right)\right|\leq C\left\|(x_j)_{j=1}^k\right\|_{mid,p}\cdot\left\|(y_j^*)_{j=1}^k\right\|_{w,p^*}$$
 for every $k \in \mathbb{N}$, all $x_j\in E$ and $y_j^* \in F^*$, $j=1,\ldots,k$.\\
{\rm (v)} There is a constant $C>0$ such that $$\sum_{j=1}^\infty \left|y_j^*\left(T\left(x_j\right)\right)\right|\leq C \left\|(x_j)_{j=1}^\infty\right\|_{mid,p}\cdot\left\|(y_j^*)_{j=1}^\infty\right\|_{w,p^*}$$ for all $\left(x_j\right)_{j=1}^\infty \in \ell_{p}^{mid}(E)$ and $(y_j^*)_{j=1}^\infty \in \ell_{p^*}^{w}(F^*)$.

Moreover,
$$\|T\|_{{\cal D}_p^{mid}}:= \|\widetilde{T}\| = \inf\{C: {\rm (iii)~holds}\} = \inf\{C: {\rm (iv)~holds}\}= \inf\{C: {\rm (v)~holds}\}.$$
\end{theorem}

\begin{theorem}\label{teoideal}The classes $\left(\Pi_{p;q}^{mid}, \|\cdot\|_{\Pi_{p;q}^{mid}}\right)$, $\left(W_{p;q}^{mid}, \|\cdot\|_{W_{p;q}^{mid}}\right)$  and  $\left({\cal D}_{p}^{mid}, \|\cdot\|_{{\cal D}_p^{mid}} \right)$ are Banach operator ideals.
\end{theorem}

\begin{proof} We use the notation and the language of \cite{botelhocampos}.
All involved sequence classes are linearly stable. Comparing Definition \ref{defop} and \cite[Definition 2.1]{botelhocampos}, a linear operator $T$ is  mid $(p;q)$-summing if and only if is $\left(\ell_{q}^{mid}(\cdot);\ell_p(\cdot)\right)$-
summing. Since $\ell_{q}^{mid}(\mathbb{K})  \stackrel{1}{\hookrightarrow} \ell_p = \ell_p(\mathbb{K}),$ from \cite[Theorem 2.6]{botelhocampos} it follows that $\Pi_{p;q}^{mid}$ is a Banach operator ideal. The other cases are similar.
\end{proof}

The following characterizations of weak and strong mid $p$-spaces complement the ones proved in \cite[Theorems 3.7 and 4.5]{karnsinha}.

\begin{theorem}\label{revis} The following are equivalent for a Banach space $E$ and $1\leq p<\infty$:\\
{\rm (a)} $E$ is a weak mid $p$-space.\\
{\rm (b)} $\Pi_{p}^{mid}(E;F)=\Pi_{p}(E;F)$ for every Banach space $F$.\\
 {\rm (c)} $\Pi_{p}^{mid}(E;\ell_{p})=\Pi_{p}(E;\ell_{p}) = {\cal L}(E;\ell_{p})$.\\
{\rm (d)} $W_{p}^{mid}(F;E)={\cal L}(F;E)$ for every Banach space $F$.\\
{\rm (e)} $id_E\in W_{p}^{mid}(E;E)$.
\end{theorem}

\begin{proof} The implications (a) $\Longrightarrow$ (b), (d) $\Longrightarrow$ (e) $\Longrightarrow$ (a), and
(b) $\Longrightarrow$ the first equality in (c) are obvious. Let us see that the first equality in (c) implies (a): Given $x^* =(x^*_k)_{k=1}^\infty \in \ell_{p}^{w}(E^*)$, the identification $\ell_{p}^{w}(E^*)= \mathcal{L}(E,\ell_{p})$ (see the proof of Proposition \ref{propcomp}) yields that the map
$$S_{x^*}\colon E \longrightarrow \ell_p~,~S_{x^*}(x)=(x_k^*(x))_{k=1}^\infty,$$
is a bounded linear operator. By the definition of $\ell_p^{mid}(E)$,
$$(S_{x^*}(x_n))_{n=1}^\infty = ((x^*_k(x_n))_{k=1}^\infty)_{n=1}^\infty \in \ell_{p}(\ell_{p}), $$
for every  $(x_n)_{n=1}^\infty \in \ell_{p}^{mid}(E)$. This means that $S_{x^*}\in \Pi_{p}^{mid}(E;\ell_{p})$, hence $S_{x^*}\in \Pi_{p}(E;\ell_{p})$ by assumption, for every $x^*=(x^*_k)_{k=1}^\infty \in \ell_{p}^{w}(E^*)$. So, given $(x_n)_{n=1}^\infty \in \ell_{p}^{w}(E)$, it follows that $(S_{x^*}(x_n))_{n=1}^\infty=((x^*_k(x_n))_{k=1}^\infty)_{n=1}^\infty
\in \ell_{p}(\ell_{p})$ for every $x^*=(x^*_k)_{k=1}^\infty \in
\ell_{p}^{w}(E^*)$; proving that $(x_n)_{n=1}^\infty \in
\ell_{p}^{mid}(E)$.

That (a) is equivalent to the second equality in (c) is precisely Theorem \ref{coinc}(i).

To complete the proof, let us check that (a) $\Longrightarrow$ (d): Let $T\in {\cal L}(F;E)$ and $(x_j)_{j=1}^\infty \in \ell_{p}^{w}(F)$ be given. The linear stability of $\ell_p^{w}(\cdot)$  and the assumption give  $\left(T\left(x_j\right)\right)_{j=1}^\infty \in \ell_p^w(E)=\ell_p^{mid}(E).$ This proves that $T\in W_{p}^{mid}(F;E)$.
\end{proof}

The corresponding characterizations of strong mid $p$-spaces are less interesting. We state them just for the record:

\begin{theorem}\label{revis2} The following are equivalent a Banach space $F$ and $1\leq p<\infty$:\\
{\rm (a)} $F$ is a strong mid $p$-space.\\
{\rm (b)} $\Pi_{p}^{mid}(E;F)={\cal L}(E;F)$ for every Banach space $E$.\\
{\rm (c)} $id_F\in \Pi_{p}^{mid}(F;F)$.\\
{\rm (d)} $F$ is a subspace of $L_p(\mu)$ for some Borel measure $\mu$.
\end{theorem}

Recall that an operator ideal $\cal I$ is:\\
$\bullet$ {\it Injective} if $u \in {\cal I}(E,F)$ whenever $v \in {\cal L}(F,G)$ is a metric injection ($\|v(y)\| = \|y\|$ for every $y \in F$) such that $v \circ u \in {\cal I}(E,G)$.\\
$\bullet$ {\it Regular} if $u \in {\cal I}(E,F)$ whenever $J_F \circ u \in {\cal I}(E,F^{**})$, where $J_F \colon F \longrightarrow F^{**}$ is the canonical embedding.

\begin{proposition} The operator ideal $\Pi_{p;q}^{mid}$ is injective and the operator ideals $W_{p;q}^{mid}$ and ${\cal D}_{p}^{mid}$ are regular.
\end{proposition}
\begin{proof} The injectivity of $\Pi_{p;q}^{mid}$ is clear. To prove the regularity of $W_{p;q}^{mid}$, let $(y_j)_{j=1}^\infty \subseteq F$ be such that $(J_F(y_j))_{j=1}^\infty \in \ell_p^{mid}(F^{**})$. We have \begin{equation}\label{zequ}(y_n^{***}(J_F(y_j)))_{j,n=1}^\infty \in \ell_p(\ell_p) {\rm ~~ for~ every~} (y^{***}_n)_{n=1}^\infty \in B_{\ell_p^w(F^{***})}.
\end{equation}
In order to prove that $(y_j)_{j=1}^\infty \in \ell_p^{mid}(F)$, let $(y^*_n)_{n=1}^\infty \in  B_{\ell_p^w(F^{*})}$ be given. Then
\begin{equation}\label{fequ}\sum_{n=1}^\infty |J_F(y)(y_n^*)|^p = \sum_{n=1}^\infty |y_n^*(y)|^p \leq 1 {\rm ~~for~ every~} y \in B_F.\end{equation}
Let us see that, defining $y_n^{***} := J_{F^*}(y_n^*) \in F^{***}$ for each $n$, we have $(y^{***}_n)_{n=1}^\infty \in B_{\ell_p^w(F^{***})}$. To accomplish this task, let $y^{**} \in B_{F^{**}}$ be given. By Goldstine's Theorem, there is a net $(y_\lambda)_\lambda$ in $B_F$ such that $J_F(y_\lambda) \stackrel{w^*}{\longrightarrow} y^{**}$, that is,
\begin{equation}\label{sequ}y^*(y_\lambda) = J_F(y_\lambda)(y^*) \longrightarrow y^{**}(y^*) {\rm ~~  for~ every~} y^* \in F^*.
\end{equation}
From (\ref{fequ}) it follows that $ \sum\limits_{n=1}^\infty |y_n^*(y_\lambda)|^p \leq 1$ for every $\lambda$, in particular
\begin{equation}\label{tequ}\sum\limits_{n=1}^k |y_n^*(y_\lambda)|^p \leq 1 {\rm ~~for ~every~} k {\rm ~ and~ every~} \lambda.
 \end{equation}
 On the other hand, from (\ref{sequ}) we have $|y_n^*(y_\lambda)|^p \stackrel{\lambda}{\longrightarrow} |y^{**}(y_n^*)|^p$ for every $n$, hence $$\sum\limits_{n=1}^k |y_n^*(y_\lambda)|^p \stackrel{\lambda}{\longrightarrow} \sum\limits_{n=1}^k|y^{**}(y_n^*)|^p$$ for every $k$. So,
\begin{align*}\sum_{n=1}^\infty |y_n^{***}(y^{**})|^p &= \sum_{n=1}^\infty |J_{F^*}(y_n^*)(y^{**})= \sum_{n=1}^\infty |y^{**}(y_n^*)|^p = \sup_k \sum_{n=1}^k |y^{**}(y_n^*)|^p \\&= \sup_k \lim_\lambda \sum\limits_{n=1}^k|y_n^{*}(y_\lambda)|^p \stackrel{(\ref{tequ})}{\leq} 1,
\end{align*}
for every $y^{**} \in B_{F^{**}}$. This proves that $(y^{***}_n)_{n=1}^\infty \in B_{\ell_p^w(F^{***})}$. From (\ref{zequ}) we get
$$ (y_n^*(y_j))_{j,n=1}^\infty = (J_F(y_j)(y_n^*))_{j,n=1}^\infty= ([J_{F_{*}}(y_n^*)](J_F(y_j)))_{j,n=1}^\infty= (y_n^{***}(J_F(y_j)))_{j,n=1}^\infty \in \ell_p(\ell_p).$$
This holds for arbitrary $(y^*_n)_{n=1}^\infty \in  B_{\ell_p^w(F^{*})}$, which allows us to conclude that $(y_j)_{j=1}^\infty \in \ell_p^{mid}(F)$. Thus far we have proved that $(y_j)_{j=1}^\infty \in \ell_p^{mid}(F)$ whenever $(J_F(y_j))_{j=1}^\infty \in \ell_p^{mid}(F^{**})$. Now the regularity of $ W_{p;q}^{mid}$ follows easily.\\
\indent An adaptation of the argument above shows that $(y_j)_{j=1}^\infty \in \ell_p\langle F \rangle$ whenever $(J_F(y_j))_{j=1}^\infty \in \ell_p\langle F^{**}\rangle$. The regularity of ${\cal D}_p^{mid}$ follows.
\end{proof}
\begin{remark}\rm The final part of the proof above also proves that the ideal ${\cal D}_p$ of Cohen strongly $p$-summing operators is regular. We also know that it is surjective because it is the dual of the injective ideal $\Pi_{p^*}$ \cite{cohen73}.
\end{remark}

It is clear from the definitions that $\Pi_{p,r}^{mid} \circ W_{r,q}^{mid} \subseteq \Pi_{p,q}$ for $q \leq r \leq p$. Next we show that the equality holds if $p = q$, what gives a new factorization theorem for absolutely $p$-summing operators:

\begin{theorem}\label{thefact}
Every absolutely $p$-summing linear operator factors through
absolutely and weakly mid $p$-summing linear operators, that is,
$\Pi_p = \Pi_p^{mid} \circ W_{p}^{mid} $.
\end{theorem}

\begin{proof} We already know that $\Pi_p^{mid} \circ W_{p}^{mid} \subseteq \Pi_p$.
Let $u \in \Pi_p(E;F)$. By Pietsch's factorization theorem (\cite[Corollary 1, page 130]{df} or \cite[Theorem 2.13]{djt}), there are a Borel-Radon measure $\mu$ on $(B_{E^*},w^*)$, a closed subspace $X$ of $L_p(\mu)$ and an operator $\widehat{u}\colon X \rightarrow F$ such that
the following diagram commutes ($i_E$ and $j_p$ are the canonical operators and $j_p^E$ is the restriction of $j_p$ to $i_E(E)$):
\begin{equation*}
\begin{gathered}
\xymatrix@C15pt@R23pt{
E \ar@/_/[d]_*{i_E} \ar[rr]^*{u} &  & F \\
i_E(E) \ar[rr]^*{j_p^E} &  & X \ar@/_/[u]_*{\widehat{u}}\\
{\cal C}(K) \ar@{}[u]|*{\bigcap} \ar[rr]^*{j_p} & & L_p(\mu) \ar@{}[u]|*{\bigcap}.
}
\end{gathered}
\end{equation*}

Let $(x_j)_{j=1}^\infty \in \ell_p^{w}(E)$. By the continuity of $i_E$ and the linear stability of $\ell_p^w(\cdot)$, we have $\left(i_E(x_j)\right)_{j=1}^\infty \in \ell_p^{w}\left(i_E(E)\right)$. Since $j_p$ is absolutely $p$-summing, it follows that $\left(j_p^E(i_E(x_j))\right)_{j=1}^\infty \in \ell_p(X) \subseteq \ell_p^{mid}(X)$, proving that
$j_p^E \circ i_E \in W_{p}^{mid}(E,X)$. Now, let $(y_j)_{j=1}^\infty \in \ell_p^{mid}(X)$. As $X$ is a closed subspace of $L_p(\mu)$, from Theorem \ref{coinc}(ii) we have $(y_j)_{j=1}^\infty \in \ell_p(X)$. Thus, as $\widehat{u}$ is bounded and $\ell_p(\cdot)$ is linearly stable,
$\left(\widehat{u}(y_j)\right)_{j=1}^\infty \in \ell_p(F)$,
proving that $\widehat{u} \in \Pi_p^{mid}(X,F)$.
\end{proof}

\begin{corollary}Let $p > 1$, $u \in {\cal L} (E,F)$ and $v \in {\cal L}(F,G)$. If $u^*$ is absolutely mid $p^*$-summing and $v^*$ is weakly mid $p^*$-summing, then $v \circ  u$ is Cohen strongly $p$-summing.
\end{corollary}
\begin{proof}Denoting, as usual, by ${\cal I}^{\rm dual}$ the ideal of all operators $u$ such that $u^* \in {\cal I}$, we have
$$(W_{p^*}^{mid})^{\rm dual} \circ (\Pi_{p^*}^{mid})^{\rm dual} \subseteq (\Pi_{p^*}^{mid} \circ W_{p^*}^{mid})^{\rm dual} = \Pi_{p^*}^{\rm dual} = {\cal D}_p,$$
where the inclusion is clear, the first equality follows from Theorem \ref{thefact} and the second from \cite{cohen73}.
\end{proof}

We finish this section solving a question left open in the last section of \cite{karnsinha}. The authors prove the following characterization (see \cite[Theorem 4.4]{karnsinha}): an operator
$T\in\mathcal{L}(E,F)$ is weakly mid $p$-summing if and only if
$S\circ T\in\Pi_{p}(E,\ell_{p})$ for every $S \in \mathcal{L}(F,\ell_{p})$. They define
\[lt_p(T)=\sup\{\pi_{p}(S\circ T):  S\in \mathcal{L}(F,\ell_{p})\
\text{and}\ \|S\|\leq1\},\]
and prove that $\left(W_{p}^{mid}, lt_p(\cdot)\right)$ is a normed operator ideal. The question whether or not this ideal is a Banach ideal is left open there, and now we solve it in the affirmative:

\begin{proposition}\label{propcomp} $lt_p(T)=\|T\|_{W_{p;q}^{mid}}$  for every $T\in W_{p}^{mid}(E;F)$, hence $\left(W_{p}^{mid}, lt_p(\cdot)\right)$ is a Banach operator ideal.
\end{proposition}

\begin{proof}
Let $T\in W_{p}^{mid}(E;F)$ and $S\in \mathcal{L}(F,\ell_{p})$ with $\|S\|\leq 1$. Here we use that the spaces $\ell_{p}^{w}(F^*)$
and $\mathcal{L}(F,\ell_{p})$ are canonically isometrically isomorphic via the correspondence
$x^*=(x^*_k)_{k=1}^\infty \in \ell_{p}^{w}(F^*) \mapsto$
$S_{x^*}\in \mathcal{L}(F,\ell_{p})$,
$S_{x^*}(x)=(x_k^*(x))_{k=1}^\infty$
 \cite[Proposition 8.2(2)]{df}. So there exists
$(y_k^*)_{k=1}^\infty \in B_{\ell_{p}^{w}(F^*)}$ such that
$S(y)=(y_k^*(y))_{k=1}^\infty$ for every $y\in F$.
Thus
\begin{equation*}
\left(\sum_{j=1}^{\infty} \left\Vert S\circ T \left(x_j\right) \right\Vert_{p} ^{p}\right)  ^{1/p} = \left( \sum_{j=1}^{\infty}\sum_{k=1}^{\infty}
\left\vert y_{k}^{\ast}(T\left(x_j\right))\right\vert ^{p}\right) ^{1/p} \leq \|T\|_{W_{p;q}^{mid}}\cdot\left\|(x_j)_{j=1}^{\infty}\right\|_{w,p},
\end{equation*}
for every $(x_j)_{j=1}^\infty \in \ell_{p}^{w}(E)$. Therefore
$S\circ T\in \Pi_{p}(E;\ell_{p})$ and $\pi_{p}(S\circ T)\leq\|T\|_{W_{p;q}^{mid}}.$
From
\begin{equation*}
\left(\sum_{j=1}^{\infty} \sum_{n=1}^{\infty}
\left\vert y_{n}^{\ast}(T(x_{j}))\right\vert ^{p}\right) ^{1/p} =
\left(
\sum_{j=1}^{\infty}
\left\Vert S\circ T(x_{j})\right\Vert_{p} ^{p}\right)  ^{1/p}\leq\pi_{p}(S\circ T)\cdot\left\|(x_j)_{j=1}^\infty\right\|_{w,p},
\end{equation*}
we get $\|T\|_{W_{p;q}^{mid}}\leq \pi_{p}(S\circ T),$
proving that $lt_p(T)=\|T\|_{W_{p;q}^{mid}}$. The second assertion follows now from Theorem \ref{teoideal}.
\end{proof}

\section{Infinite dimensional Banach spaces formed by non-summing operators}
\label{sec3}

We say that the subset $A$ of an infinite dimensional vector space $X$ is {\it lineable} if $A \cup \{0\}$ contains an infinite dimensional subspace. If $A \cup \{0\}$ contains a closed infinite dimensional subspace than we say that $A$ is {\it spaceable} (see \cite{bams} and references therein).

Let us give a contribution to this fashionable subject. Improving a result of \cite{diogo}, in \cite{timoney} it is proved, among other things, that if $E$ is an infinite dimensional superreflexive Banach space, then, regardless of the infinite dimensional Banach space $F$, there exists an infinite dimensional Banach space formed, up to the null operator, by non-$p$-summing linear operators from $E$ to $F$. Very little is known for spaces of operators on non-superreflexive spaces. We shall give a contribution in this direction.

The next lemma is left as an exercise in \cite{df} (Exercise 9.10(b)). We give a short proof the sake of completeness.

\begin{lemma}\label{lemapietsch} An operator ideal $\cal I$ is injective if and only if the following condition holds:  if $u \in {\cal I}(E;F)$, $v \in {\cal L}(E;G)$ and there exists a constant $C >0$ (eventually depending on $E,F,G,u$ and $v$) such that  $\|v(x)\| \leq C\|u(x)\|$ for every $x \in E$, then $v \in {\cal I}(E;G)$.
\end{lemma}

\begin{proof} Assume that $\cal I$ is injective and let $u \in {\cal I}(E;F)$, $v \in {\cal L}(E;G)$ be such that $\|v(x)\| \leq C\|u(x)\|$ for every $x \in E$. This inequality guarantees that the map
$$w \colon u(E) \subseteq F \longrightarrow G~,~w(u(x)) = v(x), $$
is a well defined continuous linear operator. Considering the canonical metric injection $J_G \colon G \longrightarrow \ell_\infty(B_{G^*})$, by the extension property of $\ell_\infty(B_{G^*})$ \cite[Proposition C.3.2]{pietschlivro} there is an extension $\widetilde{w} \in {\cal L}(F;\ell_\infty(B_{G^*}))$ of $J_G \circ w$ to the whole of $F$. From $\widetilde{w} \circ u = J_G \circ v$ we conclude that $J_G \circ v$ belongs to $\cal I$, and the injectivity of $\cal I$ gives $v \in {\cal I}(E;G)$.  The converse is obvious.
\end{proof}

Henceforth, all Banach spaces are supposed to be infinite dimensional. Recall that a sequence in a Banach space $E$ is {\it overcomplete} if the linear span of each of its subsequences is dense in $E$ (see, e.g. \cite{chalendar, fonf}). We need a weaker condition:

\begin{definition}\label{defnew}\rm A sequence in a Banach space $E$ is {\it weakly overcomplete} if the closed linear span of each of its subsequences is isomorphic to $E$.
\end{definition}

\begin{example}\label{examp}\rm The sequence $(e_j)_{j=1}^\infty$ formed by the canonical unit vectors is a weakly overcomplete unconditional basis in the spaces $c_0$ and $\ell_p$, $1 \leq p < +\infty$ \cite[Proposition 4.45]{fabian}.
\end{example}

\begin{proposition}\label{lemmalin} Let $(\cal I, \|\cdot\|_{\cal I})$ be a normed operator ideal, ${\cal J}$ be an injective operator ideal and suppose that $F$ contains an isomorphic copy of a space $X$ with a weakly overcomplete unconditional basis. If  ${\cal I}(E;X) - {\cal J}(E;X)$ is non-void, then ${\cal I}(E;F) - {\cal J}(E;F)$ is spaceable (in $({\cal I}(E;F), \|\cdot\|_{\cal I})$).
\end{proposition}

\begin{proof} 
Let $(e_{n})_{n=1}^\infty$ be a weakly overcomplete unconditional basis of $X$ with unconditional basis constant $\varrho$.
Split $\mathbb{N} = \bigcup\limits_{j=1}^\infty A_j$ into infinitely many infinite pairwise disjoint subsets. For each $j \in \mathbb{N}$, define $X_j =  \overline{{\rm span}\{e_n : n \in A_j\}}$ and let $P_j \colon X \longrightarrow X_j$ be the canonical projection. It is known that $\|P_j\| \leq \varrho$ \cite[Corollary 4.2.26]{meg}. For $x_j \in X_j$ we have $P_i(x_j) = \delta_{ij}x_j$ because the sets $(A_j)_{j=1}^\infty$ are pairwise disjoint.  Let $I_j \colon X \longrightarrow X_j$ be an isomorphism, $T_j \colon X_j \longrightarrow X$ denote the formal inclusion and $T \colon X \longrightarrow F$ be an isomorphism into. Let $u \in {\cal I}(E;X) - {\cal J}(E;X)$. Defining
$$u_j \colon E \longrightarrow F~,~u_j = T\circ T_j \circ I_j \circ u, $$
we have $u_j \in {\cal I}(E,F)$. Using that $\cal J$ is injective, $u \notin {\cal J}(E;X)$ and
\begin{align*}\|u_j(x)\| &= \|T(T_j \circ I_j \circ u(x))\| \geq \frac{1}{\|T^{-1}\|}\|T_j \circ I_j \circ u(x)\| \geq \frac{1}{\|T^{-1}\|\cdot \|I_j^{-1}\|}\|u(x)\|
\end{align*}
for every $x \in E$, we conclude by Lemma \ref{lemapietsch} that each $u_j \notin {\cal J}(E;F)$. In particular, $u_j \neq 0$. Let $Y:=\overline{{\rm span}\{u_j : j \in \mathbb{N}\}}^{\|\cdot\|_{\cal I}} \subseteq {\cal I}(E;F)$. Given $0 \neq v \in Y$, let $(v_n)_{n=1}^\infty \subseteq {\rm span}\{u_j : j \in \mathbb{N}\}$ be such that $v_n \stackrel{\|\cdot\|_{\cal I}}{\longrightarrow} v$. For each $n$, write $v_n = \sum\limits_{j=1}^\infty a_j^n u_j$, where $a_j^n \neq 0$ for only finitely many $j$'s. Let $x_0 \in E$ be such that $v(x_0) \neq 0$. It is plain that $v(E) \subseteq T(X)$, so $T^{-1}(v(x_0)) \neq 0$, and in this case there is $k \in \mathbb{N}$ such that $P_k(T^{-1}(v(x_0))) \neq 0$. Since $\|\cdot\| \leq \|\cdot\|_{\cal I}$, we have $v_n(x) \longrightarrow v(x)$ for every $x \in E$. So,
\begin{align*} a_k^n \,T_k(I_k(u(x_0))) & = \sum_{j=1}^\infty P_k\left(a_j^n\,T_j(I_j(u(x_0))) \right) = \sum_{j=1}^\infty P_k\left(T^{-1}(a_j^n\,T(T_j(I_j(u(x_0))))) \right)\\
& = P_k \circ T^{-1}\left( v_n(x_0) \right) \longrightarrow P_k \circ T^{-1} \circ v(x_0) \neq 0.
\end{align*}
It follows that
$$0 \neq T \circ P_k \circ T^{-1} \circ v(x_0) = \lim_n T\left(  a_k^n \,T_k(I_k(u(x_0)))\right) = \lim_n a_k^n u_k(x_0) = \left(\lim_n a_k^n\right) u_k(x_0) . $$
Calling $\lambda := \lim\limits_n a_k^n \neq 0$, we have
\begin{align*}\|u_k&(x)\| =  \frac{1}{|\lambda|}\cdot \lim_n\|a_k^n u_k(x)\|\leq 
\frac{\|T\|}{|\lambda|}\cdot \lim_n\left\| P_k\left(\sum_{j=1}^\infty T_j \circ I_j\circ u(a_j^nx)\right)\right\| \\&\leq \frac{\varrho\|T\|}{|\lambda|}\cdot \lim_n \left\|\sum_{j=1}^\infty T_j \circ I_j \circ u(a_j^nx)\right\| \leq \frac{\varrho\|T\|\cdot\|T^{-1}\|}{|\lambda|}  \cdot \lim_n\left\|T\left(\sum_{j=1}^\infty T_j \circ I_j \circ u(a_j^nx)\right)\right\|\\
& = \frac{\varrho\|T\|\cdot\|T^{-1}\|}{|\lambda|}  \cdot \lim_n \left\|\sum_{j=1}^\infty a_j^n u_j(x)\right\| =  \frac{\varrho\|T\|\cdot\|T^{-1}\|}{|\lambda|} \|v(x)\|
\end{align*}
for every $x \in E$. Since $u_k$ does not belong to the injective ideal $\cal J$,  it follows from Lemma \ref{lemapietsch} that $v \notin {\cal J}(E;F)$. This proves that $Y \subseteq ({\cal I}(E;F) - {\cal J}(E;F) \cup \{0\})$.

Given $n \in \mathbb{N}$, scalars $a_1, \ldots, a_n$ such that $\sum\limits_{j=1}^n a_j u_j = 0$ and $k \in \{1, \ldots, n\}$, let $x_k \in E$ be such that $u_k(x_k) \neq 0$ (recall that $u_k \neq 0$). From
\begin{align*}0 & = \left\|\sum\limits_{j=1}^n a_j u_j(x_k)\right\|\geq \frac{1}{\|T^{-1}\|}\left\|\sum\limits_{j=1}^n a_j (T_j \circ I_j \circ u)(x_k)\right\| \\&\geq \frac{1}{\varrho\|T^{-1}\|}\left\|P_k\left(\sum\limits_{j=1}^n a_j (T_j \circ I_j \circ u)(x_k)\right)\right\| = \frac{1}{\varrho\|T^{-1}\|}\left\| a_k (T_k \circ I_k \circ u)(x_k)\right\|\\
&\geq \frac{1}{\varrho\|T^{-1}\|\cdot\|T\|}\left\|T( a_k (T_k \circ I_k \circ u)(x_k))\right\|=  \frac{1}{\varrho\|T^{-1}\|\cdot\|T\|} |a_k| \cdot \|u_k(x_k)\|,
\end{align*}
it follows that $a_k = 0$,  proving that the set $\{u_j : j \in \mathbb{N}\}$ is linearly independent.
\end{proof}

\begin{remark}\rm (a) Proposition \ref{lemmalin} is not a consequences of \cite[Proposition 2.4]{timoney} because we are not assuming neither that $({\cal I}\cap{\cal J})(E;F)$ is not closed in ${\cal I}(E;F)$ nor that ${\cal I}(E;F)$ is complete.\\
(b) A result related to Proposition \ref{lemmalin}, with different assumptions, has appeared recently in \cite[Theorem 3.5]{espanhois}.
\end{remark}

Recall that Space($\cal I$) denotes the class of all Banach spaces $E$ such that the identity operator on $E$ belongs to the operator ideal $\cal I$ (cf. \cite[2.1.2]{pietschlivro}).

\begin{theorem} Let $E$ be isomorphic to a subspace of $L_1(\mu)$ for some Borel measure $\mu$, let $F$ contain an isomorphic copy of $\ell_1$ and let $({\cal I}, \|\cdot\|_{\cal I})$ be a Banach operator ideal such that $\ell_1 \in {\rm Space}({\cal I})$. Then there exists an infinite dimensional Banach space formed, up to the null operator, by non-$1$-summing linear operators from $E$ to $F$ belonging to $\cal I$.
\end{theorem}

\begin{proof} By Theorem \ref{coinc}(ii), $id_E \in \Pi_1^{mid}(E;E)$. Since $id_E$ fails to be 1-summing, because $E$ is infinite dimensional, by Theorem \ref{thefact} we have $id_E \notin W_1^{mid}(E;E)$. From Theorem \ref{coinc}(i), there is a non-1-summing linear operator $u \colon E \longrightarrow \ell_1$. Of course $u \in {\cal I}(E;\ell_1)$. Taking into account that the canonical unit vectors form a weakly overcomplete unconditional basis of $\ell_1$ (Example \ref{examp}) and that the ideal of absolutely $p$-summing linear operators is injective, from Proposition \ref{lemmalin} we have that ${\cal I}(E;F) - \Pi_1(E;F)$ is spaceable. The completeness of $({\cal I}(E;F), \|\cdot\|_{\cal I})$ finishes the proof.
\end{proof}

Examples of Banach operator ideals $\cal I$ for which $\ell_1 \in {\rm Space}({\cal I})$ are the following: separable operators, completely continuous operators, cotype 2 operators, absolutely $(r,q)$-summing operators with $\frac{1}{r} \leq \frac{1}{q} - \frac12$ \cite[Corollary 8.9]{df} (in particular, absolutely $(r,1)$-summing operators for every $r \geq 2$).\\

\noindent{\bf Acknowledgement.} We thank Professor A. Pietsch for pointing out Lemma \ref{lemapietsch} to us.

\bigskip

\noindent Faculdade de Matem\'atica\\
Universidade Federal de Uberl\^andia\\
38.400-902 -- Uberl\^andia -- Brazil\\
e-mail: botelho@ufu.br\\

\noindent Departamento de Ci\^{e}ncias Exatas\\
Universidade Federal da Para\'iba\\
58.297-000 -- Rio Tinto -- Brazil\\
e-mail: jamilson@dcx.ufpb.br and jamilsonrc@gmail.com\\

\noindent Departamento de Matem\'atica\\
Universidade Federal da Para\'iba\\
58.051-900 -- Jo\~ao Pessoa -- Brazil\\
e-mail: joedsonmat@gmail.com

\end{document}